\documentclass[11pt,twoside]{amsart}
\usepackage{amssymb,latexsym,amsthm,amsmath,mathtools,amsfonts,hyperref,fancyhdr,comment, tabularht, tabularx, longtable, mathrsfs, cleveref,booktabs, color, geometry,bm}
\usepackage[
singlelinecheck=false % <-- important
]{caption}

\usepackage{geometry}
\hypersetup{
	colorlinks=true,
	linkcolor=blue,
	filecolor=magenta,
	urlcolor=cyan,
	citecolor=red,
}

\long\def\symbolfootnote[#1]#2{\begingroup%
	\def\thefootnote{\fnsymbol{footnote}}\footnote[#1]{#2}\endgroup}

\newcommand{\Z}{\ensuremath{\mathcal{Z}}}

\def \N {{\mathbb{N}}}
\def \Z {{\mathbb{Z}}}

\def \al {{\alpha}}

\newtheorem*{theorem*}{Theorem}
\newtheorem{theorem}{Theorem}[section]
\newtheorem{cor}[theorem]{Corollary}

\newtheorem{lemma}[theorem]{Lemma}
\newtheorem{ex}[theorem]{Example}

\newtheorem*{ex*}{Example}

\newtheorem{pro}[theorem]{Proposition}
\makeatletter
\def\imod#1{\allowbreak\mkern10mu({\operator@font mod}\,\,#1)}
\makeatother

\usepackage{tikz}
\tikzstyle{vertex}=[circle, draw, inner sep=0pt, minimum size=6pt]

\numberwithin{equation}{section}
\newcommand{\ignore}[1]{}

\newcommand{\mynote}[1]{}

\begin{document}
	
	\title{CLT-Groups with Cyclic or Abelian Subgroups}
	
	\author{Khyati Sharma}
	\address{Shiv Nadar Institution of Eminence, NH-91, Dadri, Gautam Buddha Nagar}
	\email{khyatisharma0907@gmail.com}
	
	\author{A. satyanarayana Reddy}
	\address{Shiv Nadar Institution of Eminence, NH-91, Dadri, Gautam Buddha Nagar}
	\email{satya.a@snu.edu.in}
	
	\today 
	\begin{abstract}
A finite group is called a {\em CLT-group} if it contains a subgroup corresponding to every divisor of the order of the group. It is said to be a {\em Cyclic (Abelian) CLT group} if it contains a cyclic (abelian) subgroup corresponding to every proper divisor of the order of the group. A natural number is said to be a {\em CCLT (ACLT) number} if every group of that order is a cyclic (abelian) CLT group. In this work, we classify all CCLT and ACLT numbers and study various properties of Cyclic (Abelian) CLT groups. We also show that the classes of CCLT and ACLT groups are contained in the class of supersolvable groups. Moreover, we introduce the function {\em CCLT-degree} on the set of non-cyclic finite groups and study the properties of this function.

\end{abstract}
\maketitle

\section{Introduction}
\par One of the important theorems of group theory is {\em Lagrange's theorem}, which states that for a finite group $G$, the order of a subgroup divides the order of the group. The converse of Lagrange's theorem is not true because the alternating group with $4$ letters of order 12 has no subgroup of order $6$. This motivates the definition of a CLT-group and a CLT number. A finite group $G$, which satisfies the converse of Lagrange's theorem, is called a {\em CLT-group}. A natural number $n$ is said to be a {\em CLT number} if every group of order $n$ is a CLT-group. Berger~\cite{Berger} gave a number theoretic result to check when a natural number $n$ is CLT by using the properties of prime factorization of $n$. More results on CLT numbers can be seen in \cite{Baskaran1976, Nganon, Marius}.  CLT-groups are also studied by many authors for further information (see \cite{ballester2007some, MR2139516, brandl1987clt, brandl1989character, gagen1977note, heineken1995groups, liu2012clt, salunke2009converse}). 
\par The study of CLT groups is important as some interesting characterizations of other (familiar) classes of groups are obtained in terms of CLT groups. Holmes \cite{holmes1966characterization} proved that a finite group $G$ is nilpotent if and only if for every divisor of the order of $G$, there exists a normal subgroup of that order. W.E. Deskins~\cite{deskins1968characterization} showed that a group $G$ is supersolvable if and only if every subgroup of $G$ is CLT. Humphreys~\cite{MR0387403} proved that a group of odd order is supersolvable if all its factor groups are CLT. Hall in $1937$ proved that a finite group is solvable if and only if every Sylow subgroup has a complement. Therefore, every CLT group is solvable. 
\par There has been growing interest in this field. The following recent contributions heightened focus on this area. In $2009$, Li, He, Nong, and Zhou~\cite{li2009hall} studied a new class of CLT-groups. They introduced the following definition. A subgroup $H$ of a group $G$ is called {\em Hall normally embedded} in $G$ if $H$ is a Hall subgroup of the normal closure $H^G$. They completely characterized those groups $G$, which contain a Hall normally embedded subgroup for every divisor $d$ of the order of $G$. In $2012$, Liu, Li, and He~\cite{liu2012clt} studied and characterized one more class of CLT-groups known as $C$-groups. A group $G$ is called a {\em $C$-group} if, for each divisor $d$ of the order of $G$, $G$ contains a subgroup $H$ of order $d$ such that $H$ is either normal or abnormal in $G$. Shen, Brandl, Shi, and Chen~\cite{shen2020clt} defined and characterized a few more classes of CLT-groups using some normality conditions. Recently, T\u{a}rn\u{a}uceanu~\cite{tuarnuauceanu2024clt} proved that for a composite number $d$, there always exists a solvable group having no subgroup of order $d$. 
\par In this paper, we define two more classes of CLT-groups. Their definitions are as follows. A finite group $G$ is said to be a {\em Cyclic CLT (CCLT) or Abelian CLT (ACLT) group}, if $G$ contains a cyclic or abelian subgroup for every proper divisor $d$ of the order of $G$, respectively. A natural number $n$ is said to be a {\em CCLT (ACLT) number} if every group of order $n$ is a CCLT (ACLT) group. In this article, we characterize all CCLT and ACLT numbers. Also, we prove that the classes of CCLT and ACLT groups are properly contained in the class of supersolvable groups. The following theorem is the main result of this paper.
 \begin{theorem}\label{thm:ACLT}
A natural number $n$ is an ACLT number if and only if one of the following conditions holds
\begin{enumerate}
    \item $n$ is an abelian number.
    \item $n=pq$, where $p<q$ and $p|q-1$.
    \item $n=p^m$, where $p$ is prime and $m\in \{0,1,2,3,4\}$.
    \item $n=p^2q$, where $p|q-1,\ p^2\nmid q-1$ and $p$ and $q$ are distinct primes except $n=12$.
    \end{enumerate}
\end{theorem}
By definition, every CCLT group is ACLT, but the converse is not true, as a dihedral group $D_6$ of order $12$ is ACLT but not CCLT. So, the class of CCLT groups is properly contained in the class of ACLT groups. Therefore, the following Theorem, which is the characterization of CCLT numbers, is a special case of the above Theorem.
\begin{theorem}\label{thm:CCLT}
A natural number $n$ is a CCLT number if and only if either $n$ is a cyclic number that is $(n, \varphi(n))=1$ or $n=pq$, where $p$ and $q$ are prime numbers, need not be distinct.
\end{theorem}
\par A new approach to study group-theoretic problems is to study them statistically and get probabilistic results, which help to understand
the asymptotic behavior. In this context, the probability of a random subgroup of $G$ to be cyclic is introduced by T\u{a}rn\u{a}uceanu and T\'{o}th~\cite{CyclicityDegree}, in $2015$ known as the {\em cyclicity degree} of $G$. Recently, T\u{a}rn\u{a}uceanu~\cite{tuarnuauceanu2024clt} introduced and studied the probability that a random subgroup of a finite group $G$ is CLT, referring to it as the {\em CLT-degree} of $G$. In this paper, we similarly define and investigate the probability that a random subgroup of a finite non-cyclic group $G$ is CCLT, which we call the {\em CCLT-degree} of $G$.\\
The organization of the article is as follows. Some notations and definitions are set in Section~\ref{sec-notation}. In Section~\ref{sec:CCLT}, a proof of Theorem~\ref{thm:CCLT} is given. We also characterize CCLT groups except for non-abelian $p$-groups in Proposition~\ref{pro:CCLTorder}. Additionally, a few more properties and examples of CCLT groups are given in this section. Section~\ref{sec:ACLT}, deals with a proof of Theorem~\ref{thm:ACLT}, various properties of ACLT groups, and some of their concrete examples for a better understanding of their structure. In section~\ref{CCLT-degree} we define and study the function {\em CCLT-degree} which measures the probability that a random subgroup to be CCLT. We conclude the article in Section~\ref{sec:conclusion}  by giving some future directions.
\section{Notation and Preliminaries}\label{sec-notation}
Throughout this paper, all groups are finite. The notations $C_n$, $D_n$, and $Dic_n$ denote the cyclic group of order $n$, dihedral group of order $2n$ and dicyclic group of order $4n$, respectively. The Euler totient function is denoted by $\varphi$. The functions $\tau(n)$ and $\omega(n)$ denote the number of divisors and the number of distinct prime divisors of $n$, respectively. The set $W(n)$ denotes the collection of prime divisors of $n$. Here, $p,q$, and $r$ are distinct prime numbers, and the number of Sylow $p$-subgroups of a group $G$ is denoted as $n_p(G)$. If $G$ is a group and $g\in G$, then $o(g)$ denotes the order of $g$ in $G$. If $G$ is a group and $H$ is a subgroup of $G$, then $N_G(H)$ and $C_G(H)$ denote the normalizer and centralizer of $H$ in $G$, respectively. A natural number $n$ is said to be an {\em cyclic (abelian) number} if every group of order $n$ is a cyclic (abelian) group. A group $G$ is known as {\em metacyclic (metabelian)} if $G$ contains a cyclic (abelian) normal subgroup $H$ such that $G/H$ is cyclic (abelian). A non-cyclic group is said to be {\em minimal non-cyclic} if all its proper subgroups are cyclic, and a non-abelian group is said to be {\em minimal non-abelian} if all its proper subgroups are abelian. These groups are completely characterized in \cite[Proposition~2.8]{jafari2017number} and \cite{miller1906number, redei1947schiefe}.

\section{CCLT Groups}\label{sec:CCLT}
\par We begin this section with the proof of the Theorem~\ref{thm:CCLT}. After this, we characterize CCLT groups except for non-abelian $p$-groups and discuss some properties related to the number of subgroups of CCLT groups. Also, we show that CCLT groups are closed with respect to subgroups and quotients, and give the necessary and sufficient conditions for the direct product of CCLT groups to be CCLT.\\
Here are the following observations about CCLT numbers.
By using the definition of a CCLT group, we can say that every cyclic group is a CCLT group thus, every cyclic number is a CCLT number. A natural number $n$ is a cyclic number if and only if $(n, \varphi(n))=1$ by \cite{pakianathan2000nilpotent}.\\ 
\textbf{Proof of Theorem~\ref{thm:CCLT}:} If $n$ is a cyclic number, $n=p^2$ or $n=pq$, where $p$ and $q$ are distinct prime numbers, then it is easy to see that $n$ is a CCLT number.
 We prove the converse by contradiction, that is, if $n$ is neither a cyclic number nor of the form $p^2$ or $pq,$ then there exists a non-CCLT group of that order.
 \par \noindent Let $n=p^k,$ where $k\ge 3$. Then the group $G=C_p\times C_p\times C_{p^{k-2}},$  does not contain a cyclic subgroup of order $p^{k-1}$. Suppose $n$ is a square free number of the form $n=p_1 p_2\cdots p_k,$ where $k\ge 3$, which is not a cyclic number, then there exists $p_i,p_j\in\{p_1,p_2,\ldots,p_k\}$ such that $p_i|p_{j}-1$. If we consider the group $G=H\times C_{\frac{n}{p_ip_j}},$ where $H$ is a non-abelian group of order $p_ip_j$, then $G$ does not contain a cyclic subgroup of order $p_ip_j$. Finally, if $n$ is of the form $p_1^{a_1}p_2^{a_2}\cdots p_k^{a_k}$, where $a_1\cdot a_2\cdots a_k\ge 2$, then $G=\left(C_{p_j}\times C_{p_j^{a_{j-1}}}\right)\times C_{ \frac{n}{p_j^{a_j}} }$, where $a_j\ge 2$ does not contain a cyclic subgroup of order $p_j^{a_j}$. Hence, the result follows.\\
The following are a few nontrivial examples of CCLT groups. 

\begin{ex}\label{ex:CCLT:1}
 Dihedral group $D_n$ and dicyclic group $Dic_n$ are CCLT if and only if either $n$ is a prime number or $n=2^k$, where $k\in\N$.
 \end{ex}
 \begin{ex}\label{ex:CCLT:2}
The group $G=C_{p^{n-1}}\rtimes C_p$ is a CCLT group, where $n\in \N$.
\end{ex}
\begin{ex}\label{ex:CCLT:3}
 The semi-dihedral group $S_{2^n}$ and generalized quaternion group $Q_{2^n}$ of order $2^n$ are CCLT groups these groups can be found in \cite{gn}.
 \end{ex}
  
 \begin{pro}\label{pro:metacylcic}
  Every CCLT group is metacyclic and hence supersolvable. 
 \end{pro}
 \begin{proof}
     Let $G$ be a CCLT group of order $n$, and $p$ be the smallest prime divisor of $n$. Then by definition, $G$ contains a cyclic subgroup $H$ of order ${n}/{p}$. Since the subgroup of the smallest prime index of a group is normal, then $H$ is normal in $G$. Also, $G/H$ is cyclic. Therefore, $G$ is metacyclic. By \cite[Theorem~1.2]{bray1982between}, every metacyclic group is supersolvable. Thus, every CCLT group is also supersolvable.
 \end{proof}
 Here, an interesting observation is that most of the non-cyclic CCLT groups are minimal non-cyclic. It is not true in general as the group $C_p\times C_{p^{n-1}}$ is a CCLT group of order $p^n$, its subgroup $C_p\times C_{p^{n-2}}$ of order $p^{n-1}$ is not cyclic. Next, we prove that a CCLT group which is neither abelian nor a non-abelian $p$-group is minimal non-cyclic. To prove this, we first prove the following Lemma and Theorem~\ref{thm:Zgroup}.
 \begin{lemma}\label{thm:sylow}
 Every Sylow $p$-subgroup of a CCLT group that is not a $p$-group is cyclic.
 \end{lemma}
\begin{proof}
The proof is immediate by using the Sylow theorem and the definition of the CCLT group.
\end{proof}
In the literature, the groups satisfying the hypothesis of Lemma~\ref{thm:sylow} are called {\em $Z$-groups}. For more examples of $Z$-groups, we refer \cite{gn}. Also every $Z$-group is metacyclic by \cite[Theorem~9.4.3]{hall2018theory}. Dihedral group $D_{15}$ is an example of $Z$-group, which is not a CCLT group. Therefore, the class of CCLT groups except $p$-groups is properly contained in the class of $Z$-groups. Hence, if we exclude groups of prime power order, then the class of CCLT groups has the following relationship with the other known classes of groups.
$$\mbox{Cyclic groups} \subset \mbox{CCLT groups} \subset \mbox{$Z$-groups}\subset \mbox{Metacyclic groups}\subset \mbox{Supersolvable groups}.$$
\begin{theorem}\label{thm:Zgroup}
 Any two subgroups of the same order in a $Z$-group are isomorphic. Moreover, every proper subgroup of a CCLT group that is not a $p$-group is cyclic.
 \end{theorem}
\begin{proof}
 Let $G$ be a $Z$-group of order $mn$. Then by using the Theorem~$11$ of Chapter~$5$ \cite{hans}, $G$ has the following presentation: 
 \begin{equation}\label{eq:ab}
  G=\langle a,b|a^m=b^n=e, bab^{-1}=a^r, \gcd((r-1)n,m)=1,r^n\equiv 1 \;\pmod m\rangle.
 \end{equation}
Let $H$ and $K$ be the subgroups of $G$ of the same order $d$. If $d|m$ or $d|n,$ then we are done as both $H$ and $K$ are cyclic subgroups of order $d$. Suppose $d=ts$, where $t, s>1$, $t|m$ and $s|n$. Then $$H=\langle a^{k_1},b^{k_2}\rangle,K=\langle a^{v_1},b^{v_2}\rangle,\;\;\mbox{where}\;\;o(a^{k_1})=o(a^{v_1})=t, o(b^{k_2})=o(b^{v_2})=s.$$
Let us define a map $\psi:H\to K$ as $\psi(a^{k_1})=a^{v_1} $ and $\psi(b^{k_2})=b^{v_2}$. Then, by using relationships given in Equation~(\ref{eq:ab}), it is easy to see that $\psi$ is an isomorphism. Hence, the result holds.
\end{proof}
\begin{cor}\label{cor:cyclic}
Every CCLT group that is neither cyclic nor a $p$-group is minimal non-cyclic.
\end{cor}
The following is the classification for CCLT groups except for non-abelian $p$-groups.
\begin{pro}\label{pro:CCLTorder} 
If $G$ is a CCLT group of order $n$ which is not a non-abelian $p$-group, then
 $$G\cong \begin{cases}
               C_n & \mbox{if $G$ is cyclic,}\\
               C_p\times C_{p^{k-1}} & \mbox{if $G$ is abelian but not cyclic,}\\
             \langle a, b| a^q=b^{p^r}=1, b^{-1}ab=a^s \rangle & \mbox{if $G$ is non-abelian,
        and $n=p^rq$, where $p|(q-1)$}\\
             & \mbox{and $r, s\in \N, q\nmid (s-1), q|(s^p-1).$}
              \end{cases}$$ 

 \end{pro}
\begin{proof}
\par \noindent If $G$ is cyclic, then $G\cong C_n$. If $G$ is an abelian group that is not cyclic, then by the Fundamental theorem of finitely generated abelian groups $G\cong C_p\times C_{p^{k-1}}$, where $n=p^k$. By Theorem~\ref{thm:Zgroup} and \cite[Proposition 2.8]{jafari2017number}, we can deduce that the order of every CCLT group which is neither cyclic nor a $p$-group is $p^rq$, where $p, q$ are primes and $p|q-1.$ Further, there exists a unique non-cyclic CCLT group of that order having the following presentation $$\langle a, b| a^q=b^{p^r}=1, b^{-1}ab=a^s \rangle,$$ where $r, s\in \N, q\nmid (s-1), q|(s^p-1)$.
This completes the proof.
\end{proof}     
The next result is related to the number of non-isomorphic CCLT groups except for non-abelian $p$-groups.
\begin{pro}\label{number:CCLT}
Let $G$ be a CCLT group of order $n$, which is not a non-abelian $p$-group, and let $G_{CCLT}(n)$ denotes the number of non-isomorphic CCLT groups of order $n$. Then
\begin{enumerate}
    \item  If $n$ is not a prime power, then 
$$G_{CCLT}(n)=\begin{cases}
               2 & \mbox{if $n=p^rq$ and $p|(q-1)$,}\\
               1 & \mbox{otherwise.}
              \end{cases}$$ 
\item If $n=p^k$, then $G_{CCLT}(n)=G(n)$ for $k=1$ or $2$, where $G(n)$ denotes the number of non-isomorphic groups of order $n$. Also
$G_{CCLT}(2^3)=4$ and

$$G_{CCLT}(p^k)\geq\begin{cases}
               6 & \mbox{if $p=2$ and $k>3$,}\\
               3 & \mbox{if $p\not=2$ and $k\geq 3$.}
              \end{cases}$$    

\end{enumerate}
\end{pro}
\begin{proof} 
Since every cyclic group is CCLT, then for every natural number $n$, there exists a CCLT group of order $n$.
    If $G$ is not a $p$-group, then by Proposition~\ref{pro:CCLTorder}, one can check that for $n=p^rq$, where $p|(q-1)$ we have $G_{CCLT}(n)=2$ otherwise it is one.
    \par \noindent If $G$ is a $p$-group, then for $n=p$ and $p^2$ by Theorem~\ref{thm:CCLT} we have $G_{CCLT}(n)=G(n)$. Also, it is easy to see that $G_{CCLT}(2^3)=4$. Moreover, for $p=2$ and $k> 3$ the semi direct product $C_{2^{k-1}}\rtimes C_{2}$ gives four CCLT groups, and $Dic_{2^{k-2}}$ is also a CCLT group of order $2^k$ by Example~\ref{ex:CCLT:1}. Therefore for $k>3, G_{CCLT}(2^k)\geq 6$.
    \par \noindent If $p$ is an odd prime then $C_{p^{k-1}}\rtimes C_{p}$ gives two CCLT groups of order $p^k$, where $k\geq 3$. Thus $G_{CCLT}(p^k)\geq 3$. Hence, the result holds.
\end{proof}
\begin{theorem}\label{thm:nosubgroup}
 Let $G$ be a CCLT group of order $n$, which is not a non-abelian $p$-group, and let $s(G)$ be the number of subgroups of $G$. Then
 $$s(G)=\begin{cases}
               \tau(n) & \mbox{if $G$ is cyclic,}\\
                2+(p+1)(k-1) & \mbox{if $G$ is abelian and $n=p^k$,}\\
               2r+q+1 & \mbox{if $n=p^rq, p\not=q$ and $p|(q-1).$}
              \end{cases}$$
\end{theorem}
\begin{proof}
\par \noindent If $G$ is cyclic then it is a well known fact that $s(G)=\tau (n)$ (see~\cite{dummit1991abstract}).
However, when $G$ is non-cyclic abelian group, then $G\cong C_p\times C_p^{k-1}$ by Proposition~\ref{pro:CCLTorder}. Then by using \cite[Theorem~A]{aivazidis2020finite} it is immediate that $s(G)=2+(p+1)(k-1)$. 
\par \noindent On the other hand if $G$ is non-abelian, then Corollary~\ref{cor:cyclic} and Proposition~\ref{pro:CCLTorder} yields that $|G|=p^rq$ and all proper subgroups of $G$ are cyclic. By using a result given in \cite[Page No.~33]{taunt_1949} it can be concluded that $G$ has a unique  cyclic subgroup $H$ of order $p^{r-1}q.$ Since any subgroup of $G$ of order $p^{r-1}q$ is maximal, therefore all other proper subgroups of $G$ are contained in $H$ except the subgroups of order $p^r.$ Now it is easy to see that the number of subgroups of $H$ is $\tau (p^{r-1}q)=2r.$ Also, number of elements of order $p^r$ in $G$ is $p^rq-p^{r-1}q,$ which implies that number of subgroups of order $p^r$ in $G$ is $q.$ Hence $S(G)=2r+q+1$.
\end{proof}
\begin{cor} Let $G$ be a CCLT group which is not a non-abelian $p$-group and let $c(G)$ be the number of cyclic subgroups of $G$. Then
 $$c(G)=\begin{cases}
               \tau(n) & \mbox{if $G$ is cyclic,}\\
                (k-1)p+2 & \mbox{if $G$ is abelian and $n=p^k$,}\\
               2r+q & \mbox{if $n=p^rq, p\not=q$ and $p|(q-1).$}
              \end{cases}$$
 \end{cor} 
 \begin{proof}
 The proof is immediate by using \cite[Theorem~1]{c(G)for_abelian_groups} and Theorem~\ref{thm:nosubgroup}.
 \end{proof}
 Next, we prove that CCLT groups are closed under passage to subgroups and quotients. Also, we discuss the necessary and sufficient conditions for the direct product of CCLT groups to be CCLT, as it is not true in general.

\begin{theorem}\label{thm:CCLTsubgroup}
Every subgroup of a CCLT group is CCLT.  
\end{theorem}
\begin{proof}
Let $G$ be a CCLT group. Then, there are the following two cases.
\begin{enumerate}
\item Let $G$ be a $p$-group of order $p^n$. Then, one can check the result holds for $n=1$ and $2$ as $p$ and $p^2$ are CCLT numbers by Theorem~\ref{thm:CCLT}. Thus, from now onwards, assume that $n\ge 4$. We prove the result by the method of contradiction. Let $k$ be the largest such that $G$ has a non-CCLT subgroup $H$ of order $p^k$. Then $H$ is normal in a CCLT subgroup $N$ of $G$ of order $p^{k+1}$. Since $N$ is CCLT, so it has a cyclic subgroup of order $p^k$, let us call it $K$. Then $HK$ is a subgroup of $N$ and 
 $$|HK|=\frac{|H||K|}{|H\cap K|}=p^{2k-\al}\le p^{k+1},$$ where $\al=|H\cap K|$, which implies that $\al = k-1$ and $H$ contains a cyclic subgroup of order $p^{k-1}$. Therefore, $H$ is the CCLT group, which is a contradiction. Hence, the result holds.
 \item If G is not a $p$-group. In this case by Corollary~\ref{cor:cyclic}, we can say that every proper subgroup of $G$ is cyclic.\\ 
 Therefore, CCLT groups are closed under passage to subgroups.
 \end{enumerate}
\end{proof}
\begin{theorem}
Quotient group of a CCLT group is CCLT.
\end{theorem}
\begin{proof}
Let $G$ be a CCLT group and $N$ be a normal subgroup of $G$. Then we will prove the result casewise.
\begin{enumerate}
 \item Let $G$ be a $p$-group of order $p^n$. To prove this result, it is sufficient to consider the case $|N|=p^m$, where $m\leq n-2$. Since $G$ is a CCLT group, then it contains an element of order $p^{n-1}$ say $g$. Let $|gN|=p^k$, where $k\leq {n-m}$. Then $o({g^p}^k)\leq p^m.$ Consider, $$o({g^p}^k)=\frac {o(g)} {(o(g), p^k)}=\frac {p^{n-1}}{(p^{n-1}, p^k)}=\frac {p^{n-1}}{p^k}=p^{n-k-1}\leq p^m.$$ As a consequence  $k\geq n-m-1.$ This shows that $k=n-m-1$ or $k=n-m$. Therefore, $G/N$ is CCLT.
\item Let $G$ be not a $p$-group. 
Then by Corollary~\ref{cor:cyclic} all proper subgroups of $G$ are cyclic. If $N$ is a normal subgroup of $G$ then by Correspondence theorem, every proper subgroup of $G/N$ is cyclic. Hence, $G/N$ is CCLT.
\end{enumerate}
\end{proof}
 \begin{theorem}\label{thm:DPCCLT}
  Let $H$ and $K$ be the groups of order $m$ and $n$, respectively.
  If $H\times K$ is a CCLT group, then both $H$ and $K$ are cyclic. Conversely, if $H$ and $K$ are cyclic groups, then $H\times K$ is a CCLT group if one of the following conditions holds:
  \begin{enumerate}
   \item \label{thm:DPCCLT:1} If $\gcd(m,n)=1.$
   \item \label{thm:DPCCLT:2} If $\gcd(m,n)\ne 1,$ then $H\times K\cong C_p\times C_{p^k},$ where $p$ is a prime number and $k\in \N.$
  \end{enumerate}
 \end{theorem}
\begin{proof}
 Suppose $H\times K$ is a CCLT group, and $H$ is not a cyclic group. If $\gcd(m, n)=1$ then it is easy to see that $H\times K$ has no cyclic subgroup of order $m$. If $\gcd(m,n)\ne 1$ then $p|\gcd(m, n)$, where $p$ is a prime number. Also $m=p^rx$ and $n=p^sy$, where $\gcd(p, x)=\gcd(p, y)=1.$ Then it is easy to verify that $H\times K$ has no cyclic subgroup of order $p^{\max\{r,s\}+1}$. This is a contraction.\\
 Conversely, suppose that $H$ and $K$ are cyclic groups. If $\gcd(m,n)=1,$ then $H\times K$ is cyclic, and hence it is a CCLT group. If $\gcd(m,n)\ne 1$ then $p|\gcd(m, n)$, where $p$ is a prime number. Then $m=p^rx$ and $n=p^sy,$ where $\gcd(p,x)=\gcd(p,y)=1.$ Now our aim is to show that $xy=1$ and either $r=1$ or $s=1$. Suppose $xy\ne 1$ or both $r>1$ and $s>1$, then $H\times K$ has no cyclic subgroup of order $p^{\max\{r,s\}+1}$. This contradicts the fact that $H\times K$ is a CCLT group. Hence the result follows.
\end{proof}
% From Theorem~\ref{thm:sylow} and Theorem~\ref{thm:DPCCLT} the following result is immediate.
% \begin{cor}
%  Let $G$ be a CCLT group of order $n=p_1^{a_1}\cdot p_2^{a_2}\cdots p_k^{a_k},$ where $k\ge 2.$ Then $G$ is nilpotent if and only if $G$ is cyclic.
% \end{cor}

\section{ACLT Groups}\label{sec:ACLT}
In this section, we first discuss a proof of Theorem~\ref{thm:ACLT}, along with some examples of ACLT groups. We also prove that the order of a non-abelian ACLT group has at most two prime divisors. Moreover, we prove that the class of ACLT groups is properly contained in the class of supersolvable groups. At last, we show that ACLT groups are closed with respect to subgroups and the necessary and sufficient conditions for the direct product of ACLT groups to be ACLT.\\
We start with the following results about abelian numbers. A natural number $n$ is an abelian number if and only if it is cube-free and there is no prime power $p^k|n$ with $k\geq 1$, such that $p^k \equiv 1 \mod{q}$ for a prime $q|n$ by \cite{pakianathan2000nilpotent}. 
\begin{lemma}\label{ACLT,orderP^2q}
    A group $G$ of order $p^2q$ is ACLT if and only if either $G$ is abelian or $p|q-1$ but $p^2\nmid q-1$.
\end{lemma}
\begin{proof}
  The idea of this proof is taken from the proofs of Propositions $3.2$ and $3.3$ of \cite{kalra2019finite}. To prove this result, it is sufficient to prove that non-abelian groups of order $p^2q$, where $p|q-1$ but $p^2\nmid q-1$, contain an element of order $pq$, and for other cases, there exists a non-abelian group having no element of order $pq$. There are the following two cases:\\
  The first case is $p<q$. If $p\nmid q-1$, then $G$ is abelian, which implies that $p|q-1$. Now, there are the following possibilities. First, assume that $p^2\nmid q-1$ except in the case $p=2$ and $q=3$. There are two non-abelian groups of this type by Page No. $76-80$ of \cite{burnside_2012}. The first such group can be presented as $$G_1=\langle a,b,c : a^q=b^p=c^p=1, bab^{-1}=a^k, ca=ac, cb=bc, ord_q(k)=p\rangle.$$ Also $G_1\cong (C_q\rtimes C_p)\times C_p$. By Sylow theorem $(C_q\rtimes C_p)$ has $q$ subgroups of order $p$ and $G_1$ has no element of order $p^2$. If $(x, y)\in G_1$ has order $p$, then either both $x$ and $y$ have order $p$ or one of them is of order $p$. Thus $G_1$ has $p^2q-pq+p-1$ elements of order $p$. Since $G_1$ has a unique Sylow $q$-subgroup, so it has $q-1$ elements of order $q$. Now we can see that $G_1$ has $(p-1)(q-1)$ elements of order $pq$. The second group can be presented as $$G_2=\langle a,b : a^q=b^{p^2}=1, bab^{-1}=a^k,ord_q(k)=p\rangle.$$ Also $G_2\cong C_q\rtimes C_{p^2}$. By Sylow theorem, $G_2$ has unique Sylow $q$-subgroup and $q$ Sylow $p$-subgroups. Thus $G_2$ contains $q-1$ elements of order $q$ and $pq(p-1)$ elements of order $p^2$. By using the fact that $bab^{-1}=a^k$ and  $ord_q(k)=p$, we have $b^pa=a^{k^p}b^p=ab^p$. This implies that every element of $G_2$ can be written as $a^mb^n$, where $1\leq m\leq q$ and $1\leq n\leq p^2$. Also $(a^mb^n)^p=1$ only if $m=q$ and $n$ is a multiple of $p$. Therefore $G_2$ has $p-1$ elements of order $p$ and remaining $(p-1)(q-1)$ elements are of order $pq$. Hence, $G_2$ is ACLT. If $p^2q=12$, then the Alternating group $A_4$ of order $12$ is not ACLT.\\
  Now suppose that $p^2|q-1$. Then by Page No. $76-80$ of \cite{burnside_2012}, there are three non-abelian groups of this order. They are $G_1$ and $G_2$ as defined above, and the third group can be presented as $$G_3=\langle a,b : a^q=b^{p^2}=1, bab^{-1}=a^k,ord_q(k)=p^2\rangle.$$ Also $G_3\cong C_q\rtimes C_{p^2}$. The group $G_3$ also has $q-1$ elements of order $q$ and $pq(p-1)$ elements of order $p^2$. 
 Every element of $G_3$ can be written in the form $a^mb^n$, where $1\leq m\leq q$ and $1\leq n\leq p^2$ same as in $G_2$. By using the relation $bab^{-1}=a^k$ we have $(a^mb^n)^p=a^{(m(1+k^n+k^{2n}+\cdots k^{(p-1)n}))}b^{np}=a^{m({\frac{k^{pn}-1}{k^{n}-1}})}b^{np}$. Then $(a^mb^n)^p=1$ if and only if $n$ is a multiple of $p$ and $m$ can take any value. Therefore $G_3$ has $(p-1)q$ elements of order $p$ and no element of order $pq$. Therefore $G_3$ is not ACLT.\\
  The second case is $p>q$. If $G$ is a group of order $p^2q$, then $G$ has a unique Sylow $p$-subgroup, and the number of Sylow $q$-subgroups is either $p$ or $p^2$. If $n_q(G)=p^2$, then $G$ has $p^2q-p^2$ elements of order $q$ and $p^2$ elements of order $1, p$ and $p^2$. Thus $G$ has no element of order $pq$. If $n_q(G)=p$, then $G$ has $p(q-1)$ elements of order $q$. Since $n_p(G)=p^2$ so $G$ has $p^2$ elements of order $1, p$ and $p^2$ so remaining $p(p-1)(q-1)$ elements of order $pq$. Consider the group $$G=\langle a,b : a^{p^2}=b^{q}=1, bab^{-1}=a^k,ord_{p^2}(k)=q\rangle.$$ Then, same as previous part, we can check that every element of $G$ is of the form $(a^mb^n)$, where $1\leq m\leq p^2$ and $1\leq n\leq q$ and $G$ has $p^2(q-1)$ elements of order $q$. This shows that $n_q(G)=p^2$ and $G$ are not ACLT. Hence, the result follows.
\end{proof}
\noindent \textbf{Proof of Theorem~\ref{thm:ACLT}:} First, we prove that if $n$ satisfies the hypothesis, then $n$ is an ACLT number. If $n$ is an abelian number or $n=pq$ or $n=p^3$, then it is easy to verify that $n$ is an ACLT number. If $n=p^4$, then in order to show that $n$ is an ACLT number, it is sufficient to prove that every non-abelian group of order $p^4$ contains an abelian subgroup of order $p^3$. By \cite{miller1906number}, a group of order $p^m$ has an abelian subgroup of order $p^k$, where $k(k+1)\ge 2m$. Therefore $n=p^4$ is an ACLT number. If $n=p^2q$, where $p|q-1, p^2\nmid q-1$ is an ACLT number except $n=12$ by Lemma~\ref{ACLT,orderP^2q}.
  \par \noindent To prove the converse part, it is shown that whenever $n$ is not satisfying the hypotheses, then there exists a non-ACLT group of order $n$. Now there are the following cases.
  \begin{enumerate}
      \item Suppose $n$ is a squarefree number of the form $n=p_1\times p_2\times p_3\times\cdots \times p_m$, where $m\ge 3$. Since $n$ is not an abelian number then there exists $p_i,p_j\in \{p_1,p_2,\ldots,p_m\}$ such that $p_i|p_j-1$. Let $H$ be a non-abelian group of order $p_i\times p_j$, then the group 
  $G=H\times C_{\frac{n}{p_ip_j}}$, does not contain an abelian subgroup of order $p_ip_j$.
 \item Let $n=p^m$, where $p$ is a prime number and  $m\ge 5$. If $p=2$ then the group $(C_2)^3\rtimes C_4= \langle a,b,c,d | a^2=b^2=c^2=d^4=1, ab=ba, ac=ca, dad^{-1}=abc, dbd^{-1}=bc=cb, cd=dc \rangle$ has no abelian subgroup of order $2^4$ by \cite{GAP4}. Thus the group $((C_2)^3\rtimes C_4)\times C_{2^{m-5}}$ has no abelian subgroup of order $p^{m-1}$. If $p$ is an odd prime, then there exists a non-abelian group of order $p^5$ having no abelian subgroup of order $p^4$ by \cite{MR1502819}, let us call this group $H$. Then the group $H\times C_{p^{m-5}}$ has no abelian subgroup of order $p^{m-1}$.
\item Suppose $n=p_1^{a_1}p_2^{a_2}\dots p_k^{a_k}$ and $a_j\ge 3$ for some $j\in \{1,2,\ldots k\}$. Then $p_j^{a_j}$ is not an abelian number, so there exists a non-abelian group of order $p_j^{a_j}$ say $H$. Thus the group 
$G=H\times C_{\frac{n}{p_j^{a_j}}}$, does not contain an abelian subgroup of order $p_j^{a_j}$. From now onwards, suppose that 
$n=p_1^{a_1}p_2^{a_2}\dots p_k^{a_k}$, $k\ge 2$ and $a_j\le 2$ for all $j\in \{1,2,\ldots ,k\}$. If $k\ge 3$ then there exists $p_i$ and $p_j$ such that $p_i|p_j^{a_j}-1$ as $n$ is not an abelian number. Thus, there exists a non-abelian group $H$ of order $p_ip_j^{a_j}$. If we consider the group $G=H\times C_{\frac{n}{p_ip_j^{a_j}}}$, then $G$ does not contain an abelian subgroup of order $p_i^{a_i}p_j^{a_j}$. Finally, the problem boils down to the cases $n=p^2q$ or $p^2q^2,$ it is known that some of these numbers are non-CLT (see~\cite{Baskaran1976,Marius,Struik1,Struik2}). Hence we confine ourselves to the CLT numbers of these forms.
\item If $n=p^2q$ and $n$ is not satisfying the hypothesis, then by Lemma~\ref{ACLT,orderP^2q}, $n$ is not an ACLT number. 
\item Let $n=p^2q^2,$ where $p|q-1, p^2\nmid q-1$. Now, there are two possibilities. If $p=2$, then the dihedral group $D_{2q^2}$ of order $4q^2$ does not have an abelian subgroup of order $4q$ by definition. If $p\not =2$ then by \cite{seyyed2018simple}, consider the group $G$ presented as $$G=\langle a,b,c : a^{p^2}=b^q=c^q=1, a^{-1}ba=b^k,  a^{-1}ca=c^k, cb=bc, ord_q(k)=p\rangle$$ does not contain an abelian subgroup of order $p^2q$. Similarly, for all other CLT numbers of the form $p^2q^2$, there exists a non-ACLT group $H$ of order $p^2q$ by Lemma~\ref{ACLT,orderP^2q} such that the group $G=H\times C_q$ has no abelian subgroup of order $pq^2$.\\
This completes the proof.
 \end{enumerate}

 The following are a few nontrivial examples of ACLT groups.
 \begin{ex}\label{ex:ACLT1}
 Every minimal non-abelian, CLT group is ACLT.
 \end{ex}
 \begin{ex}\label{ex:ACLT2}
Dihedral group $D_n$ of order $2n$ is an ACLT group if and only if $n\in \{p, 2p, 2^k\}$, where $k\in \mathbb N$. 
\end{ex}

\begin{lemma}\label{pro:meta}
 Every ACLT group is metabelian.
 \end{lemma}
 \begin{proof}
     The proof is similar to the proof of Lemma~\ref{pro:metacylcic}
 \end{proof}
 \begin{lemma}\label{Agroup}
     All the Sylow subgroups of an ACLT group that is not a $p$-group are abelian.
 \end{lemma}
 \begin{proof}
     The proof is immediate by using Sylow theorems and the definition of ACLT groups.
 \end{proof}
 In the literature, the groups satisfying the hypothesis of Lemma~\ref{Agroup} are known as {\em $A$-groups}. Therefore, if we exclude $p$-groups, every ACLT group is an $A$-group. To proceed further, we need the following result on $A$-groups, which is a collection of several results or observations given by D.R. Taunt~\cite{taunt_1949}.
 \begin{theorem}[D.R.Taunt~\cite{taunt_1949}]\label{thm:Agroup}
  \begin{enumerate}
   \item \label{thm:Agroup:2} Every solvable $A$-group contains unique maximal abelian normal subgroup.
   \item \label{thm:Agroup:3} Let $Z[G]$ and $G^{\prime}$ denote the center and commutator subgroups of $G$ respectively. Then $Z[G]\cap G^{\prime}=\{e\},$ whenever $G$ is an $A$-group.
  \end{enumerate}
 \end{theorem}
  \par In the literature a subgroup $H$ of a group $G$ is {\em Carter}, if $H$ is both self-normalizing and nilpotent. Carter~\cite{Carter} proved that every finite solvable group has a Carter subgroup and all its Carter subgroups are conjugate.
 \begin{lemma}\label{lem:index}
 If $G$ is a finite, non-abelian, CLT, $A$-group of order $n$, where $\omega(n)\geq 3.$ Then for some $s\in W(n)$ all the subgroups of index $s$ in $G$ are non-abelian.
 \end{lemma}
  \begin{proof} 
  Let $S$ be the set of all non-isomorphic prime index abelian subgroups of $G$. Since every element of $S$ is a maximal subgroup of $G,$ then they are either normal or self-normalizing. Suppose the hypothesis is not true, that is for every $s\in W(n)$ there exists an abelian subgroup of index $s$ in G. Therefore $|S|\ge 3.$ Further, from Part~\ref{thm:Agroup:2} of Theorem~\ref{thm:Agroup}, $S$ contains exactly one normal subgroup. Also if any two elements of $S$ are self-normalizing, then they are Carter subgroups, thus they are conjugate to each other \cite{Carter}. This shows that $|S|\le 2.$ Hence, we got a contradiction.
\end{proof}
\begin{theorem}\label{thm:two primes}
 The order of a non-abelian, ACLT group has at most two prime divisors. 
\end{theorem}
\begin{proof}
Let $G$ be a non-abelian ACLT group of order $n$. If $\omega(n)\geq 3$, then by Lemma~\ref{lem:index} $G$ is not ACLT, which completes the proof. 
\end{proof}  
 \par The next goal is to show that the class of ACLT groups is properly contained in the class of supersolvable groups. To prove this result following properties are required. Let $G$ be a group and let $p_1>p_2>\cdots >p_r$ be the distinct prime divisor of $|G|$. Then $G$ is said to satisfy the {\em Sylow tower property} if there exist Sylow subgroups $G_{p_1}, G_{p_2},\ldots ,G_{p_r}$ corresponding to the primes $p_1, p_2,\ldots ,p_r$ respectively such that $G_{p_1}G_{p_2}\cdots G_{p_k}$ is normal in $G$ for $k=1,2,\ldots ,r$. A group $G$ is said to be {\em strictly $p$-closed} if it has a normal Sylow $p$-subgroup, say $G_p$, then $G/G_p$ is abelian of exponent dividing $p-1$. For more details, one can check \cite[Theorem 1.8, 1.9]{bray1982between}.
\begin{lemma}\label{lemma:sylownormal}
Let $G$ be an ACLT group and $q$ be the largest prime divisor of  $|G|$. Then
\begin{enumerate}
 \item \label{lemma:sylownormal:1} Sylow $q$-subgroup is normal in $G.$
 \item\label{lemma:sylownormal:2} If $K$ is the Sylow $q$-subgroup of $G,$  then $G^{\prime}\leq K.$
 \item \label{lemma:sylownormal:3} Further, if $G$ is a non-abelian group, then $C_G(K),$ the centralizer of $K$ in $G$, is the index $p$ abelian subgroup of $G.$
\end{enumerate}
\end{lemma}
\begin{proof} Proof of Part~\ref{lemma:sylownormal:1}. 
If $G$ is abelian or a $p$-group, then the result holds. Suppose $G$ is a non-abelian group and $|G| = p^{a} q^{b},$ where $p<q$ by Theorem~\ref{thm:two primes}. Since $G$ is an ACLT group, it has an abelian subgroup $H$ of index $p.$ If $K$ is a subgroup of $H$ of order $q^b,$ then $K$ is also a Sylow $q$-subgroup of $G.$ Further $H\le N_G(K)$ by using the fact that $H$ is abelian. This implies that either $N_G(K)=H$ or $N_G(K)=G.$ If $N_G(K)=H$, then $n_q(G)=p$ which is not possible from Sylow's third theorem. Hence $K$ is normal in $G.$ \\
Proof of Part~\ref{lemma:sylownormal:2}.
We observe from Part~\ref{lemma:sylownormal:1} that $G$ is isomorphic to $K\rtimes L,$ where $L$ is a Sylow $p$-subgroup of $G$. By using the definition of commutator subgroup, we can conclude that $G^{\prime}\le K$. Hence, the result follows.  \\
Proof of Part~\ref{lemma:sylownormal:3}.
 Since $H$ is abelian then $H\leq C_G(K).$ Consequently, either $C_G(K)=H$ or $C_G(K)=G.$ If $C_G(K)=G$, then $K\leq Z(G).$ But by using Part~2 we have $G^{\prime}\le K.$ Hence we got contradiction by part~\ref{thm:Agroup:3} of Theorem~\ref{thm:Agroup} and by using the fact that $G$ is non-abelian.
\end{proof}
\begin{theorem}\label{thm:super}
 Every ACLT group is supersolvable.
\end{theorem}
\begin{proof}
Let $G$ be an ACLT group. If $G$ is abelian, then the result is true. Now, assume that $G$ is a non-abelian ACLT group of order $p^a q^b,$ where $p$ and $q$ are prime numbers. If $p=q,$ then $G$ is supersolvable. If $p<q$ then by Part~\ref{lemma:sylownormal:1} of Lemma~\ref{lemma:sylownormal}, $G$ has a normal Sylow $q$-subgroup say $K$. If $L$ is a Sylow $p$-subgroup of $G,$ then $KL=G$. This implies that $K$ and $KL$ are normal in $G.$ Thus $G$ has a Sylow-tower $\{e\}\le K\le KL.$ Now by Part~\ref{lemma:sylownormal:3} of Lemma~\ref{lemma:sylownormal}, $N_G(K)/C_G(K)$ is a group of order $p,$ thus $N_G(K)/C_G(K)$ is strictly $p$-closed. Also, let $L$ be any Sylow $p$-subgroup of $G$ and $N$ be an abelian subgroup of index $q$ in $G$. Then $N_G(L)=C_G(L)=N$. Therefore $N_G(L)/C_G(L)$ is also strictly $p$-closed. Hence, the result follows from \cite[Theorem~1.12]{bray1982between}. 
\end{proof}
\par Hence, the class of ACLT groups is properly contained in the class of supersolvable groups, because the group $D_{12}$ is supersolvable but not ACLT. Next, we show that ACLT groups are closed under passage to subgroups and the necessary and sufficient conditions for the direct product of ACLT groups to be ACLT.
\begin{theorem}\label{thm: subACLT}
 Every subgroup of an ACLT group is ACLT.
\end{theorem}
We need a few lemmas before proving Theorem~\ref{thm: subACLT}. 
\begin{lemma}\label{lem:pn}
 Every subgroup of an ACLT group of order $p^n,$ where $p$ is a prime number and $n\in \N$ is ACLT. 
\end{lemma}
  \begin{proof}
The proof of this is same as the proof of Theorem~\ref{thm:CCLTsubgroup}.
  \end{proof}  
\begin{lemma}\label{lem:submaximal:1}
 Let $G$ be a finite group of order $n=p^a q^b,$ where $p$ and $q$ are prime numbers. If $G$ has abelian subgroups of order $n_1=p^{a-1} q^b$ and $n_2=p^a q^{b-1}$ then $G$ is ACLT.
\end{lemma}
\begin{proof}
Let $H$ and $K$ be abelian subgroups of $G$ order $n_1$ and $n_2$ respectively. We can easily show that if $d|n$ and $d<n$ then either $d|n_1$ or $d|n_2.$ Thus $G$ has an abelian subgroup of order $d.$ Hence $G$ is ACLT.
\end{proof}
\begin{lemma}\label{lem:submaximal:2}
 Let $G$ be an ACLT group of order $p^a q^b$, where $p<q$. Then, all maximal subgroups of $G$ are ACLT.
\end{lemma}
\begin{proof}
By Theorem~\ref{thm:super}, $G$ is supersolvable. Thus, all maximal subgroups of $G$ are of prime index by \cite[Theorem~1.7]{bray1982between}. This implies that the order of any maximal subgroup of $G$ is either $p^{a-1} q^b$ or $p^a q^{b-1}$. 
Let $M$ and $N$ be any two non-abelian subgroups of $G$ of order $p^{a-1} q^b$ and $p^a q^{b-1}.$ Since $G$ is ACLT, it has abelian subgroups of order $p^{a-1} q^b$ and $p^a q^{b-1}$ let us call them $H$ and $K$ respectively. Also $HM\subseteq G$ and $KM\subseteq G$ and $$|HM|=\frac{|H|\times |M|}{|H\cap M|}=\frac{p^{2a-2}q^{2b}}{|H\cap M|}\leq {p^aq^b}.$$ Consequently $|H\cap M|=p^{a-2}q^b,$ similarly one can see that $|K\cap M|=p^{a-1}q^{b-1}.$ Also, it is easy to see that $|H\cap N|=p^{a-1} q^{b-1}$ and $|K\cap N|=p^a q^{b-2}.$ Therefore by Lemma~\ref{lem:submaximal:1}, $M$ and $N$ are ACLT.
\end{proof}
\noindent \textbf{Proof of Theorem~\ref{thm: subACLT}:} The proof is given by induction on the row index of the Hasse diagram of $G,$ starting from the topmost row. The result is trivially true for the first row as $G$ is ACLT. Let us assume that the result is true for the $k^{th}$ row, that is all the subgroups of $G$ in the $k^{th}$ row of the Hasse diagram are ACLT. Let $H$ be a subgroup of $G$ in $(k+1)^{th}$ row. Then $H$ is a maximal subgroup of some subgroup of $G$ in the $k^{th}$ row, say $N$. By induction, hypothesis $N$ is ACLT. Therefore by using Lemma~\ref{lem:submaximal:2} $H$ is ACLT. Hence, all subgroups of $G$ are ACLT.
\begin{theorem}
 Let $H$ and $K$ be the groups with $m$ and $n$ elements, respectively. If $H\times K$ is an ACLT group, then both $H$ and $K$ are ACLT, and one of them is abelian. Conversely, if $H$ is an ACLT group and $K$ is abelian, then $H\times K$ is an ACLT group if one of the following conditions holds.
 \begin{enumerate}
  \item $H$ is an abelian group.
  \item $H$ is non-abelian and $p|n,$ then $p|m.$
 \end{enumerate}
\end{theorem}
\begin{proof}
If $H$ is a non-ACLT group, then there exists $d\in \N$ such that $d|m$ and $H$ has no abelian subgroup of order $d.$ Then $H\times K$ has no abelian subgroup of order $d^{k+1},$ where $d^k|n$ but $d^{k+1}\not|n.$ Let $H$ and $K$ be non-abelian groups. Then it is well known that $H\times K$ has no abelian subgroup of order $\frac{mn}{p},$ where $p$ is the smallest prime dividing $m.$ Therefore, $H\times K$ is ACLT then both $H$ and $K$ are ACLT and one of them is abelian. Conversely, if $H$ and $K$ both are abelian, then $H\times K$ is abelian, hence ACLT. If $H$ is non-abelian and $p|n$ but $p\not|m$ then $H\times K$ has no abelian subgroup of order $\frac{mn}{p}.$ This completes the proof.
\end{proof}
\section{CCLT-degree}\label{CCLT-degree}
Viewing group-theoretic problems through a statistical framework allows for the discovery of probabilistic patterns and helps to understand their asymptotic nature. Recently,  T\u{a}rn\u{a}uceanu~\cite{tuarnuauceanu2024clt} has defined the following function on the set of all finite groups 
$$d_{CLT}(G)=\frac{D(G)}{\tau (|G|)},$$
where $D(G)$ denotes the number of divisors $d$ of $|G|$ for which there exists a subgroup of $G$ of order $d$. This function is known as {\em CLT-degree} of $G$ and it measures the probability that a random subgroup of $G$ to be CLT.

In this sequence, we also define a function on the set of non-cyclic finite groups as 
$$d_{CCLT}(G)=\frac{D_c(G)}{\tau (|G|)-1},$$ where $D_c(G)$ denotes the number of divisors $d$ of $|G|$ for which there exists a cyclic subgroup of $G$ of order $d$. This function measures the probability that a subgroup to be CCLT. We call this function as {\em CCLT-degree} of $G$. It is easy to see that this function satisfies the following properties:
\begin{enumerate}
    \item $0<d_{CCLT}(G)\leq 1$, for any non-cyclic group $G$. Moreover, $d_{CCLT}(G) = 1$ if and only if $G$ is a CCLT group.
    \item $d_{CCLT}(G_1\times G_2)=d_{CCLT}(G_1)\times d_{CCLT}(G_2)$ if and only if $gcd(|G_1|,|G_2|)=1$.
    \item If $G$ is a $p$-group of order $p^n$, then 
    $$d_{CCLT}(G)\geq \frac{2}{n}.$$
\end{enumerate}
\begin{ex}
   If $G\cong D_{n}$, where $n=p_1^{a_1}\times p_2^{a_2}\times \cdots \times p_k^{a_k}$, and $p_1<p_2<\cdots<p_k$, then $$d_{CCLT}(D_{2n})=\begin{cases}
            \frac{\prod_{i=1}^{k}(a_i+1)+1}{2 \prod_{i=1}^{k}(a_i+1)-1} & \mbox{if $n$ is odd,}\\
                \frac{ \prod_{i=1}^{k}(a_i+1)}{(a_1+2) \prod_{i=2}^{k}(a_i+1)-1} & \mbox{if $n$ is even.}\\
              \end{cases}$$
\end{ex}
\begin{ex}
    If $G\cong Dic_{n}$, where $n=p_1^{a_1}\times p_2^{a_2}\times \cdots \times p_k^{a_k}$, and $p_1<p_2<\cdots<p_k$, then $$d_{CCLT}(Dic_{n})=\begin{cases}
               \frac{2 \prod_{i=1}^{k}(a_i+1)+1}{3 \prod_{i=1}^{k}(a_i+1)} & \mbox{if $n$ is odd,}\\
                \frac{(a_1+2) \prod_{i=2}^{k}(a_i+1)}{(a_1+3) \prod_{i=2}^{k}(a_i+1)-1} & \mbox{if $n$ is even.}\\
              \end{cases}$$
\end{ex}
\begin{theorem}
    Let $\mathscr{G}$ be the set of all non-cyclic finite groups. Then $$Im(d_{CCLT})=\{d_{CCLT}(G)~|~ G\in \mathscr{G}\}$$ is dense in $[0,1]$.
\end{theorem}
\begin{proof}
    Consider the group $G_p^{n}=\Z_p\times \Z_p\times \Z_{p^{n-2}}$, where $n\geq 2$. Then $|G_p^{n}|=p^n$ and $$d_{CCLT}(G_p^{n})=\frac{n-1}{n}.$$ Since the function $d_{CCLT}$ is multiplicative, then
    $$d_{CCLT}\bigg (\prod_{i=1}^kG_{p_i}^{n_i}\bigg)=\prod_{i=1}^{k}\frac{n_i-1}{n_i},$$ where $p_1,p_2,\ldots ,p_k$ are distinct prime numbers and $I=\{n_1,n_2,\ldots,n_k\}\subset \N$. Also
    $$S=\bigg\{\prod_{n\in I}^{}\frac{n-1}{n}~\bigg|~|I|<\infty\bigg\}\subseteq Im(d_{CCLT}).$$ Thus, it is sufficient to prove that $S$ is dense in $[0,1]$.\\
    Consider the sequence $(x_n)_{n\geq 2}\subset (0,\infty)$, where $x_n=\ln{\frac{n}{n-1}}$. Now, it is easy to see that $\lim_{n\to\infty}x_n=0$. Moreover,
    \[\lim_{n\to\infty}\frac{x_n}{1/n}=1.\]    Since the series $\sum_{n\geq 2}^{}\frac{1}{n}$ is divergent, we deduce that the series $\sum_{n\geq 2}^{}x_n$ is also divergent. Therefore, all the hypotheses of \cite[Lemma 3.1]{tuarnuauceanu2024clt} are satisfied. Therefore \[\overline {\bigg\{\sum_{n\in I}^{}x_n~\bigg |~I\subset{\N}^{*},|I|<\infty\bigg\}}=[0,\infty).\] This implies that \[\overline{\bigg\{\ln{\bigg(\prod_{n\in I}^{}\frac{n}{n-1}\bigg)~\bigg|~I\subset{\N}^{*}, |I|<\infty}\bigg\}}=[0,\infty).\] Also \[\overline{\bigg\{\prod_{n\in I}^{}\frac{n}{n-1}~\bigg|~I\subset{\N}^{*},|I|<\infty\bigg\}}=[1,\infty).\] Thus \[\overline{\bigg\{\prod_{n\in I}^{}\frac{n-1}{n}~\bigg|~I\subset{\N}^{*},|I|<\infty\bigg\}}=[0,1].\] 
    We deduce that \[\overline{S}=[0,1].\] This completes the proof.
\end{proof}
\section{Conclusion}\label{sec:conclusion}
In this paper, two more classes of CLT-groups, known as CCLT and ACLT groups, are defined. Moreover, we defined CCLT and ACLT numbers. A characterization of CCLT and ACLT numbers is also given. Figure~$1$ summarizes how CCLT and ACLT numbers interact with other known numbers. We proved that these special classes of CLT groups are properly contained in the class of supersolvable groups. In the continuation, we can define Nilpotent CLT groups, a finite group $G$ is said to be {\em Nilpotent CLT group,} if $G$ contains a nilpotent subgroup for every proper divisor $d$ of the order of $G$. A natural number $n$ is said to be a {\em Nilpotent CLT number} if every group of order $n$ is a nilpotent CLT group. As a future problem, one can explore the class of Nilpotent CLT groups and their properties. Also, in most of the examples, we notice that the quotient group of an ACLT group is also ACLT. We strongly believe that ACLT groups are closed with respect to the quotient, but we are unable to prove this.\\
 
 \begin{figure}[ht!]
 \centering
 \scalebox{0.6}{
\begin{tikzpicture}
\node  at (2,5.9) {CLT numbers};
\node  at (2,5) {ACLT numbers};
\node  at (2,2) {cyclic };
\node  at (2,1.5) {numbers};
\node  at (-0.5,2) {abelian};
\node  at (-0.4,1.5) {numbers};
\node  at (4.5,2) {CCLT};
\node  at (4.5,1.5) {numbers};

\draw (2,2) circle (4.3cm);

\draw (2,2) circle (3.7cm);
\draw (1,2) circle (2.3cm);
\draw (3.3,2) circle (2.3cm);
\draw (2,2) circle (1cm);
% \draw (0,0) -- (4,0) -- (4,4) -- (0,4) -- (0,0);
 \end{tikzpicture}}
 \caption{}
 \end{figure}
 \noindent \textbf{Acknowledgement}\\

 \noindent The first-named author is supported by the University Grant Commission (UGC), India, under the scheme UGC-SRF. We thank Dr. Rijubrata Kundu and Pragati for the helpful discussions.\\
 
\noindent\textbf{Conflict of interest}\\

\noindent The authors have no conflict of interest to declare.
  \bibliographystyle{plain}
  \bibliography{refs}

\end{document}